\documentclass[preprint,aos,noautosecdot]{imsart}
\usepackage[square,sort,comma,numbers]{natbib}
\setattribute{journal}{name}{}

\RequirePackage[OT1]{fontenc}
\RequirePackage{amsthm,amsmath,natbib,graphicx}
\RequirePackage[colorlinks,citecolor=blue,urlcolor=blue]{hyperref}
\RequirePackage{hypernat}
\usepackage{amssymb,tabularx,multicol,multirow,booktabs}
\usepackage{mhequ}
\providecommand{\nopunct}{\spacefactor \@nopunct}
\def\@nopunctsfcode{1007}

\usepackage{tikz}
\usetikzlibrary{calc, shapes, backgrounds}

\startlocaldefs
\numberwithin{equation}{section}
\theoremstyle{plain}

\endlocaldefs

\newtheorem{theorem}{Theorem}[section]
\newtheorem{lemma}[theorem]{Lemma}

\usepackage{color}
\usepackage{amssymb}
\usepackage{comment}
\usepackage{pdfpages}
\usepackage{graphicx}
\usepackage{dcolumn}

\RequirePackage[numbers]{natbib}

\numberwithin{equation}{section}

\def \be{\begin{equs}}
\def \ee{\end{equs}}

\def \N{\mathrm{N}}
\def \i{\mathbf{i}}

\def \w{w}
\def \C{\rm{Cauchy(0, 1)}}

\definecolor{darkred}{rgb}{.7,0,0}

\definecolor{darkgreen}{rgb}{0,0.7,0}

\definecolor{darkblue}{rgb}{0,0,0.7}

\begin{document}

\begin{frontmatter}
\title{An unexpected encounter  \\ with Cauchy and L\'evy }
\runtitle{An unexpected encounter with Cauchy and L\'evy}

\begin{aug}
\author{\fnms{Natesh} \snm{S. Pillai}\thanksref{t1,m1}\ead[label=e2]{pillai@fas.harvard.edu}}
\and
\author{\fnms{Xiao-Li} \snm{Meng}\thanksref{t2,m1}\ead[label=e1]{meng@fas.harvard.edu}}
\thankstext{t1}{Associate Professor, Department of Statistics, Harvard University}
\thankstext{t2}{Professor, Department of Statistics, Harvard University}
\runauthor{Pillai and Meng}
\affiliation{
Department of Statistics, Harvard University\thanksmark{m1}}
\end{aug}

\begin{abstract}
The Cauchy distribution is usually presented as a mathematical curiosity, an exception to the Law of Large Numbers, or even as  an ``Evil" distribution in some introductory courses. It therefore  surprised us when  Drton and Xiao (2016) proved the following result for $m=2$ and conjectured it for  $m\ge 3$.  Let $X= (X_1,\dots, X_m)$ and $Y = (Y_1, \dots,Y_m)$ 
be i.i.d.  $\N(0,\Sigma)$, where $\Sigma=\{\sigma_{ij}\}\ge 0$ is an $m\times m$ and \textit{arbitrary} covariance matrix with $\sigma_{jj}>0$ for all $1\leq j\leq m$. Then
\begin{equation*}
Z = \sum_{j=1}^m \w_j {X_j \over Y_j} \ \sim \  \C,\;
\end{equation*}
as long as $\vec{\w}=(\w_1,\dots, 
\w_m) $ is independent of $(X, Y)$,   $\w_j\ge 0,  j=1,\dots, m$, and $\sum_{j=1}^m\w_j=1$. In this note, we present  an elementary proof of  this conjecture for any $m \geq 2$ by linking $Z$ to a geometric characterization of  Cauchy(0,1) given in Willams (1969). This general result  is essential to the large sample behavior of Wald tests in many applications such as factor models and contingency tables. It also leads to other unexpected  results such as 
\be
\sum_{i=1}^m\sum_{j=1}^m \frac{w_iw_j\sigma_{ij}}{X_iX_j}
\sim  {\text{L\'{e}vy}}(0, 1).
\ee
This generalizes the ``super Cauchy phenomenon" that the average  of $m$ i.i.d. standard L\'evy variables (i.e., inverse chi-squared variables with one degree of freedom) has the same distribution as that of a single standard L\'evy variable multiplied by $m$ (which is obtained by taking $w_j=1/m$ and $\Sigma$ to be the identity matrix).

\end{abstract}

\end{frontmatter}
\section{Cauchy Distribution: Evil or Angel?\nopunct} Many of us may  recall the surprise  or even a  mild shock we experienced the first time we encountered the Cauchy distribution.  \textit{What does it mean that it does not have a mean?}  Surely one can always take a sample average, and surely it should converge to something by the Law of Large Numbers (LLN). But then we learned that the LLN does not apply to the Cauchy distribution. Obviously there cannot be an upper bound on how things may vary  in general, and hence it is not difficult  to imagine a distribution with infinite  variance.   But the non-existence of the mean, which is not the same as the mean is infinite, is harder to envision intuitively.  Therefore, some introductory courses (e.g., at our institution) have given the Cauchy distribution the nickname ``Evil", because it has created a few excruciating moments (no pun intended), even for some of our best young minds when they tried  hard to understand the meaning of  not having a mean. 
 
Of course gain often comes with pain, because  soon we would learn something deeper. The non-existence of the mean for Cauchy is a reflection of the fact that the sample average of an i.i.d. Cauchy sample actually does converge, except it does not converge to a conventional mean, i.e., a deterministic number. Rather, it converges trivially to a Cauchy random variable, and more surprisingly, the limiting distribution  is the same as that for each term in the entire sequence,  as indexed by the sample size. In this sense the Cauchy distribution is as nice as an angel, because probabilistically  its sample average sequence never deviates from its starting point,   a dream case for anyone who studies probabilistic behavior of a random sequence. 

Through settling a conjecture set forth in \cite{DrtonXiao14}, we prove in this article that this nice property \textit{can} hold even when the i.i.d. assumption is violated (and the terms are not trivially identical).  Specifically, let  $\Sigma=\{\sigma_{ij}\} \ge 0$  and $\sigma_{jj}>0$ for all $j=1, \ldots, m,$ and let $X,Y$ be
independent variables distributed as $\N(0,\Sigma)$. We denote the row vectors as
$X= (X_1, \dots, X_m)$ and $Y = (Y_1, \dots,Y_m)$.
Let $\vec{\w}=(\w_1, \dots, \w_m)$ be  a random vector such  that \be \label{eqn:pi}
\vec{\w}\ \bot \ \{X, Y\};\quad  \sum_{j=1}^m \w_j = 1; \quad  {\rm and}\quad  \w_j\ge 0, \ j=1,\ldots,  m.
\ee
Define the random variable
\be \label{eqn:Z}
Z = \sum_{j=1}^m \w_j {X_j \over Y_j}\;.
\ee
 
In \cite{DrtonXiao14}, the $\w_j$'s were assumed to be fixed constants, but by a conditioning argument, it is trivial to generalize from deterministic $\vec\w$ to a random $\vec\w$ as long as it is independent of $(X, Y). $  Therefore, throughout this article we will present the more general random (but independent) $\vec{\w}$ version of the results  presented in  \cite{DrtonXiao14} and in the literature with fixed $\vec{\w}$, whenever appropriate.

When $\Sigma = \sigma^2 \mathrm{I}_{m\times m}$, it is well-known that $Z$ has the standard Cauchy distribution on $\mathbb{R}$ with pdf  $\pi^{-1} (1+ z^2)^{-1}$, denoted by Cauchy$(0, 1)$; Cauchy$(\mu, \sigma)$ then denotes the distribution of $\mu+\sigma Z$.  The fascination with this result is  evident from the number of different approaches proposed in the literature to prove it, such as by characteristic functions, convolutions,  multivariate change of variables, and most recently by  a trigonometric approach \cite{Cohen12}.

Nevertheless, for arbitrary $\Sigma$ (so that the terms $X_j/Y_j$, $j=1, \dots, m$ are no longer independent in general), few would expect that $Z$ might remain to be Cauchy(0,1).  However, through simulations Drton and Xiao \cite{DrtonXiao14} suspected that this was indeed the case, and via a rather complex and indirect argument involving the Residue theorem, they were able to prove it when  $m=2$ and some cases of perfect correlation when $m>2$.  They conjectured that the result should hold for $m >2$ for an arbitrary $\Sigma$,  but their argument does not seem to be easily generalizable to the $m > 2$ case, nor was it feasible  to invoke induction because of the dependence induced by $\Sigma$. \par
  
Seriously intrigued by  the findings and the conjecture in \cite{DrtonXiao14}, we worked  for a while trying to extend their complex analytic approach. By using copulas of Cauchy distributions and also the Residue theorem, we ultimately succeeded in finding a proof for all $m \geq 2$. However, we were not satisfied by our lengthy proof because it did not provide any geometric interpretation or statistical insight.  We therefore continued to search for a simpler and more inspiring  approach. Thanks to an elegant but less well-known geometric characterization of Cauchy(0,1) given in  \cite{pitman1967} and  in \cite{williams1969cauchy, Cohen12}, we  are able to provide an elementary and geometrically appealing proof of the following result, conjectured by Drton and Xiao in \cite{DrtonXiao14}.    
\begin{theorem} \label{Thm:main}
For any $\Sigma=\{\sigma_{ij}\}\ge 0$ such that $\sigma_{jj}>0$ for all $j=1,\ldots, m$  and  $\vec{\w}$ satisfying \eqref{eqn:pi}, the random variable $Z$ defined in \eqref{eqn:Z} is distributed as Cauchy(0, 1).
\end{theorem}

A theoretical speculation from this unexpected result is that for a set of random variables $\{\xi_1, \ldots, \xi_m\}$, the dependence among them can be overwhelmed by the heaviness of their marginal tails (e.g., take $\xi_j=X_j/Y_j$) in determining the stochastic behavior of their linear combinations. We invite the reader to ponder with us whether this is a pathological phenomenon or something profound.
\section{Applications  and prior work.}
As discussed in \cite{DrtonXiao14},  the $Z$ in \eqref{eqn:Z} naturally appears in many important applications. Following \cite{DrtonXiao14}, let 
$q \in \mathbb{R}[x_1,\dots,x_m]$ be a homogeneous $m$ variate polynomial with gradient $\nabla q$. Then by the $\delta$-method, the variance of $q(X)\equiv q(X_1,\ldots, X_m)$ can be approximated by $\nabla q(X) \Sigma \nabla^{\top} q(X)$, resulting in the Wald statistics  \  
\be\label{eq:wald}
W_{q,\Sigma} (X)= {q^{2}(X) \over \nabla q(X) \Sigma \nabla^{\top} q(X)} 
\equiv \left\{\nabla [\log q(X)] \Sigma \nabla^\top [\log q(X)]\right \}^{-1}.
\ee
 That is, quoting \cite{DrtonXiao14}, ``the random variable $W_{q,\Sigma}$ appears in the large sample behavior of Wald tests with $\Sigma$ as the asymptotic covariance matrix of an estimator and the polynomial $q$ appearing in a Taylor approximation to the function that defines the constraint to be tested.''  Thus there are many applications in which a distribution theory for $W_{q,\Sigma}$ is needed. These include contingency tables \cite{glonek1993,gaffke1999,gaffke2002}, graphical models \cite[Chapter 4]{drton2009lectures}, and the testing of so-called ``tetrad constraints" in factor analysis \cite{harman1960,DrtonXiao14}. See \cite{DrtonXiao14} for more applications and an extensive list of references. 

When $X \sim \mathrm{N}(0,\Sigma)$, by the arguments presented in Section 6 of \cite{DrtonXiao14}, Theorem \ref{Thm:main} implies the following result, also conjectured in \cite{DrtonXiao14},  on the quadratic forms for Gaussian random variables.

\begin{theorem} \label{thm:mot1}
Let $\Sigma=\{\sigma_{ij}\}\ge 0$ and $\sigma^2_{jj}>0$ for all $j=1,\dots, m$, and $X=(X_{1}, \dots, X_m)\sim \N(0, \Sigma)$. If $q(x_1,\dots, x_m) = x_1^{a_1} \dots x^{a_m}_m$ with non-negative real exponents $a_1, \dots, a_k$ such that $\sum_j a_j >0$,  then
\be
W_{q,\Sigma} (X)\sim {1 \over (\sum_{j=1}^ma_j )^2} \chi^2_1,
\ee
where $\chi^2_1$ denotes a standard chi-squared variable with $1$ degree of freedom.
\end{theorem}
An obvious  surprising aspect of Theorem \ref{thm:mot1} is that the exact distribution of $W_{q,\Sigma}$ is free of $\Sigma$. A consequential but somewhat hidden surprise is  revealed by expressing Theorem \ref{thm:mot1} in the following equivalent  form.    
 \begin{theorem} \label{thm:mot2} Let $\Sigma$ be the same as in Theorem~\ref{thm:mot1}. For any $(\w_1, \dots, \w_m)$ satisfying \eqref{eqn:pi}, and  $X \sim \N(0,\Sigma)$, we have
\be\label{eq:levy}
\Big({\w_1 \over X_1}, \dots, {\w_m \over X_m}\Big) \Sigma
\Big({\w_1 \over X_1},  \dots, {\w_m \over X_m}\Big)^\top \sim  \chi^{-2}_1
\ee
where $\chi^{-2}_1$ denotes an inverse chi-squared variable with $1$ degree of freedom.
\end{theorem}

These two results  are equivalent when we observe that for $q(x_1,\dots, x_m) = x_1^{a_1}\dots x_m^{a_m}$,  $\nabla \log q= ({a_1 \over x_1}, \dots, {a_m\over x_m})$. Theorem \ref{thm:mot2} therefore is merely a re-expression of Theorem \ref{thm:mot1} using the rightmost expression of (\ref{eq:wald}), and by letting $w_j=a_j/\sum_k a_k, j=1,\dots, m$. (The generalization to random but independent $w_j$'s follows the  conditioning argument discussed previously.)

When $\Sigma = \mathrm{I}_{m \times m}$, Theorem~\ref{thm:mot2} recovers the not-so-well-known ``super Cauchy phenomenon", i.e., the average of $m$ i.i.d. $\chi^{-2}_1$ is distributed exactly as $m$ times a single $\chi^{-2}_1$, by taking $w_j=m^{-1}$ for all $j$'s in Theorem~\ref{thm:mot2}. The $\chi^{-2}_1$ distribution is also known as the standard L\'evy distribution (with location parameter equal to $0$ and scale parameter $1$; see \cite{applebaum}, p. 33). This result can easily be verified  via the characteristic function of the $\chi^{-2}_1$ distribution, 
 \be
 \phi(t) = e^{-\sqrt{-2\i t}},
 \ee because $\phi^m(t/m)=\phi(mt)$.  We call  this  ``super Cauchy phenomenon" because it says that for an i.i.d. L\'evy sample of size $m$, their average is $m$ times more variable than any one of them, exceeding the case of Cauchy where the average has the same variability as a single variable.   That is, if we denote the $\alpha$-percentile of the average of $m$ i.i.d. samples by $p^{(m)}_\alpha$, then for the Cauchy sample we have $p^{(m)}_\alpha=p_\alpha^{(1)}$, but for the L\'evy sample, we have $p^{(m)}_\alpha=m p_\alpha^{(1)}$.

Clearly the characteristic function approach does not apply when $\Sigma$ is not diagonal, but nevertheless Theorem  \ref{thm:mot2} says that the above distributional result generalizes when $\Sigma$ goes beyond the diagonal case.  At the first glance, result (\ref{eq:levy}) might  seem to be a wishful  thinking by a novice to probability or algebra,  who (mistakenly) treats  $X^{-1}\Sigma X^{-\top}$, where $X^{-1}\equiv\{X_1^{-1}, \ldots, X_m^{-1}\}$, as $(X^{}\Sigma^{-1} X^{\top})^{-1}$, which would then permit  him to use the usual standardization trick by letting $Z=X\Sigma^{-1/2} \sim \N(0, I)$.  However, this would have led him to guess that the left hand side of (\ref{eq:levy}), when $w_j=m^{-1}, j=1,\dots, m$,  distributes as $m^{-2}(\sum_j Z_j^2)^{-1}$, which would then be $m^{-2}\chi_m^{-2}$, not $\chi^{-2}_1$.   

It is instructive to express the left hand side of (\ref{eq:levy}) as the average of $m^2$ terms, when we take $w_j=m^{-1}$:
\be\label{eq:levy1}
\frac{1}{m^2} \sum_{i=1}^m\sum_{j=1}^m \frac{\sigma_{ij}}{X_iX_j}
\sim  \chi^{-2}_1.
\ee
This is a rather remarkable result because  the left-hand side can  only be made invariant (algebraically) to the variances $\sigma_{jj}$ by expressing $X_j=\sqrt{\sigma_{jj}}X_j'$ with variance of $X_j'$ equal to $1$, but not to the correlations  $\rho_{ij}=\sigma_{ij}/\sqrt{\sigma_{ii}\sigma_{jj}}$. Yet, (\ref{eq:levy}) says that it is actually a pivotal quantity for $\{\rho_{ij}\}$.

There are also some works on multivariate Cauchy densities that are relevant to our problem.  For instance, in \cite{marsaglia65} and \cite{hinkley1969}, the authors studied the distribution theory for the ratio of two Gaussian random variables. Ferguson \cite{ferg1962} derived a general result for the characteristic function of a multivariate Cauchy distribution. In \cite{jam2008}, the authors studied a generalization of the bivariate Cauchy distribution. McCullagh \cite{mccullagh1992} showed that it is natural to parametrize the family of Cauchy $(\mu, \sigma)$ distributions in the complex plane, as this location-scale family is closed under M{\"o}bius transformations. Finally, we remark that Drton and Xiao \cite{DrtonXiao14} did not prove the $m=2$ case of Theorem \ref{Thm:main} directly; instead, they first proved of a special case of Theorem \ref{thm:mot1} (with $q(x_1,x_2) = x_1^{a_1}x_2^{a_2}$), generalizing a similar result of Glonek \cite{glonek1993}. Drton and Xiao then obtained the conclusion of Theorem \ref{Thm:main} for $m=2$ as a corollary using  change of variables.
\section{Proof.}
The key idea of our proof relies on the following result from \cite{pitman1967},
as reformulated in \cite{williams1969cauchy}. The proof of this lemma as given in \cite{pitman1967} is short, and it also relies on the Residue theorem for contour integration. Very recently, the author of \cite{Cohen12} gave a geometric proof  for the case of $m = 2$.
\begin{lemma} \label{thm:cauchywilres}
Let $\Theta_1 \sim \mathrm{Unif}(-\pi,\pi]$,  and $\{w_1,\ldots, w_m\}$ be independent of $\Theta_1$, where $w_j \ge 0$ and $\sum_j w_j=1.$ \ Then for  any $\{u_1, \ldots, u_{m}\}$, where $u_j \in \mathbb{R}$,
\be
\sum_{j=1}^m \w_j\tan(\Theta_1 + u_j) \sim \C.
\ee
\end{lemma}
Intuitively, if $\Theta_1 \sim\mathrm{Unif}(-\pi, \pi]$, then for any constant $u_j$,  $(\Theta_1+u_j) \mod(2\pi)  \sim\mathrm{Unif}(-\pi, \pi]$, and hence $\tan(\Theta_1 + u_j) \sim \C$.  The significance of this 
Lemma is that any convex combination of these \textit{dependent} $\C$ variables is still distributed as $\C$.  As we shall see below, Theorem \ref{Thm:main} is a direct consequence of this remarkable result.
We first prove Theorem \ref{Thm:main} when $\Sigma$ is strictly positive definite, i.e., $\Sigma>0$, and then invoke  a limiting argument  to cover the cases with $\Sigma\ge 0$. 
\begin{proof}[Proof of Theorem \ref{Thm:main}]
 When $\Sigma>0$, we  write $\Sigma^{-1}=\{b_{ij}\}. $
The joint density of $(X,Y)$ then can be written as\be 
f_{X,Y}(x,y)=K \exp\Big \{-{1 \over 2}\Big(\sum_{j=1}^m b_{jj} (x^2_j+y_j^2) + 2\sum_{j \neq k} b_{jk} (x_jx_k+y_j y_k)\Big)\Big \},
\ee
where $x,y \in \mathbb{R}^m$ and $K$ is a constant that depends only on $m$ and $\Sigma$. 
Let us make the transformation $(X_j,Y_{j}) = (R_{j}\sin(\Theta_j),R_j\cos(\Theta_j))$, where
$0 \leq R_j < \infty$ and $\Theta_j \in (-\pi, \pi]$. Then the joint density of $R=\{R_1,\ldots, R_m\}$ and $\Theta=\{\Theta_1, \ldots,\Theta_m\}$  is
\be 
\hspace{-0.75cm} f_{R, \Theta}(r, \theta) \propto \exp\Big \{-{1 \over 2}\Big(\sum_{j=1}^m b_{jj} r_j^2 + 2\sum_{j \neq k} b_{jk} r_jr_k \cos(\theta_j - \theta_k)\Big)\Big \} \prod_{j=1}^m r_j, \label{eqn:J1}
\ee
for $r \in [0,\infty)^m$ and $\theta \in (-\pi, \pi]^m$.
 The term $\prod_{j=1}^m r_j$ in Equation \eqref{eqn:J1} is the Jacobian of the $(X, Y) \rightarrow (R, \Theta)$ transformation. \par
 \begin{figure}
 \centering
 \begin{tikzpicture}[scale=0.75]
 \draw [gray!60,->] (-5,0) -- (5,0);
 \draw [gray!60,->](0,-4) -- (0,5);
   \draw [gray!30](-pi,pi/2) -- (pi,pi/2);
   \draw[gray!30](-pi/2,pi) -- (-pi/2,-pi);
 \draw [dashed]      (-pi,-pi)--(pi,pi);
 \draw       (pi,pi/2)--(-pi/2,-pi);
 \draw       (-pi/2,pi)--(-pi,pi/2);
 \draw[fill=black] (0,0) circle (0.1cm) ;
 \draw[fill=black] (pi,0) circle (0.1cm) ;
 \draw[fill=black] (-pi,0) circle (0.1cm) ;
 \draw[fill=black] (0,pi) circle (0.1cm) ;
 \draw[fill=black] (0,-pi) circle (0.1cm);
  \draw[fill=black] (-pi/2,0) circle (0.1cm);
    \draw[fill=black] (0,pi/2) circle (0.1cm);
  \draw (5.0, 0.5) node[anchor=east] {$\Theta_j$};
   \draw (0, 5) node[anchor=west] {$U_j$};
    \draw (1.0, -0.5) node[anchor=east] {$(0,0)$};
  \draw (4.7, -0.5) node[anchor=east] {$(\pi,0)$};
 \draw (-3.2, -0.5) node[anchor=east] {$(-\pi,0)$};
  \draw (-0.6, -0.5) node[anchor=east] {$(-\frac{\pi}{2},0)$};
  \draw (0, 3.6) node[anchor=west] {$(0,\pi)$};
   \draw (0.8, 1.6) node[anchor=south] {$(0,\frac{\pi}{2})$};
 \draw (0,-3.6) node[anchor=west] {$(0,-\pi)$};
 \draw (4.0,3.8) node[anchor = north] {$\Theta_1 = 0$};
 \draw (4.2,2.2) node[anchor = north] {$\Theta_1 = \frac{\pi}{2}$};
 \draw (-2.0,4.0) node[anchor = north] {$\Theta_1 = \frac{\pi}{2}$};
 \end{tikzpicture}
 \caption{For every value of $\Theta_1 \neq 0$, the map $\Theta_j \mapsto U_j$ $(j\ge 2)$ is a disjoint union of two lines, and is one-to-one. The figure shows the graph of $(\Theta_j,U_j)$ when $\Theta_1 = \frac{\pi}{2}$ (plotted as two solid lines). When $\Theta_1 = 0$, $U_j = \Theta_j$ (plotted in dashed line). 
 } \label{fig:Fonetoone}
 \end{figure}
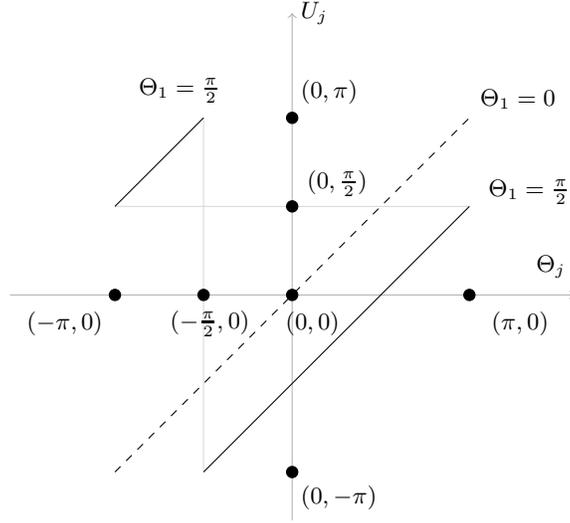
We then  make a further  transformation, $\mathcal{F}:(-\pi,\pi]^m \mapsto (-\pi,\pi]^m$, with 
$\mathcal{F}(\Theta_1, \dots, \Theta_m)=  (\Theta_1, U_2, \dots, U_m)$, where
\be \label{eq:mod} 
U_j= (\Theta_j - \Theta_1)+ 2\pi [\mathbf{1}_{\{\Theta_j - \Theta_1 \leq -\pi\}} -\mathbf{1}_{\{\Theta_j - \Theta_1 >\pi\}  }], \quad 2\le  j \leq m.
\ee 
This is a form of $U_j=(\Theta_j - \Theta_1)\mod(2\pi)$, but with the assurance  that the support of $U_j$ is $(-\pi, \pi]$ regardless of the value of $\Theta_1$, and that $U_j-U_k=(\Theta_j-\Theta_k)\mod(2\pi)$, and $\Theta_j=(\Theta_1+U_j)\mod(2\pi)$.
The map $\mathcal{F}$ is one-to-one as shown in Figure \ref{fig:Fonetoone}.
 Furthermore, the points where the map $\mathcal{F}$ is \emph{not} differentiable is contained in the set
 \be
\big \{\Theta \in (-\pi,\pi]^m: \Theta_j - \Theta_1 \in \{-\pi,\pi\} \, \mathrm{for \,\,some\,\,} j \geq 2 \big\},
 \ee 
 as can be seen from Figure \ref{fig:Fonetoone}. Clearly this set has Lebesgue measure zero. Outside this set, we have ${\partial U_j \over \partial \Theta_j} = 1$. Thus the Jacobian of the map $\mathcal{F}$  is 1 for all $\Theta \in (-\pi,\pi]^m$ except for the above measure zero set. \par
Set $U_1 \equiv 0$ and denote $U = (U_1,U_2, \dots, U_m)$.
Since $\cos(W_1) = \cos(W_2)$ for any $W_1=W_2\mod(2\pi)$,
we can write the joint density in the new coordinates as 
\be 
 f_{R, \Theta_1, U}(r, \theta_1, u) \propto 
 \exp\Big \{-{1 \over 2}\Big(\sum_{j=1}^m b_{jj} r_j^2 + 2\sum_{j \neq k} b_{jk}  r_j r_k \cos(u_k - u_j)\Big)\Big \} \prod_{j=1}^m r_j 
\ee
with $r \in [0,\infty)^m, \theta_1 \in (-\pi, \pi], u_1 = 0$ and $u_2, \dots, u_m \in (-\pi, \pi]$.
The only observations we need from the above line are: (i) $\Theta_1$ is independent of $U$ and  (ii) $\Theta_1 \sim \mathrm{Unif}(-\pi,\pi]$. 
But   $Z$ in \eqref{eqn:Z} can be written as 
\be
Z = \sum_{j=1}^m \w_j {X_j \over Y_j}= \sum_{j=1}^m \w_j \tan(\Theta_j ) 
= \sum_{j=1}^m \w_j \tan(\Theta_1 + U_j), 
\ee 
because $\tan(W_1) = \tan(W_2)$ for any $W_1=W_2\mod(2\pi)$. Since $U$ is independent of $\Theta_1$, conditional on  $U$, Lemma  \ref{thm:cauchywilres} yields that $Z\sim \C$.  It follows immediately that $Z$ is also marginally distributed as $\C$,   completing  the proof when $\Sigma>0$. 

When we relax the assumption  $\Sigma>0$ to $\Sigma\ge 0$,  $\Sigma^{-1}$ may not exist.  However,  for any $n \in \mathbb{N}$, $\Sigma^{(n)} = \Sigma +  n^{-1}\mathrm{I}_{m \times m} >0$. Let $X^{(n)} = (X^{(n)}_1,\dots, X^{(n)}_m)$ and $Y^{(n)} = (Y^{(n)}_1,\dots, Y^{(n)}_m)$ be i.i.d.  from  
$\mathrm{N}(0,\Sigma^{(n)})$. As $n \rightarrow \infty$, we have
$
 (X^{(n)}, Y^{(n)}) \rightsquigarrow (X,Y),
$
where `$\rightsquigarrow$' indicates convergence in distribution. Next the mapping $\zeta: \mathbb{R}^{2m} \mapsto \mathbb{R}$ defined by
\be 
\zeta(x,y) = \sum_{j=1}^m \w_j {x_j \over y_j}
\ee
is continuous, except when $y \in B$ with 
\be
B = \{ (y_1,\dots, y_m) \in \mathbb{R}^m: \min_{1 \leq j \leq m} |y_j| = 0 \}.
\ee
Now the result follows from the continuous mapping theorem. Indeed, since $(X^{(n)}, Y^{(n)}) \rightsquigarrow (X,Y)$, the continuous mapping theorem yields that
\be \label{eqn:weakconv}
Z^{(n)}\ = \zeta(X^{(n)}, Y^{(n)}) \rightsquigarrow  \zeta(X, Y) = Z
\ee
as $n \rightarrow \infty$, provided the points of discontinuity of $\zeta$ belong to a zero-measure set.
However, since $Y_j \sim \mathrm{N}(0,\sigma_{jj})$ where $\sigma_{jj}>0$ by our assumption, we have 
\be 
\mathbb{P}(Y \in B) = \mathbb{P}(\min_{1\leq j \leq m} |Y_j| = 0 ) \leq \sum_{j=1}^m \mathbb{P}(|Y_j| = 0) =0,
\ee
verifying \eqref{eqn:weakconv}.
By our previous argument, we know that $Z^{(n)}\sim \C$ for all $n \in \mathbb{N}$, and hence (\ref{eqn:weakconv}) implies that  $Z\sim \C.$  \end{proof}
\section*{Acknowledgements.}
We  thank Alexei Borodin, Joseph Blitzstein, Carlo Beenakker, Mathias Drton, Sukhada Fadnavis, Steven Finch, Luke Miratrix, Carl Morris, Christian Robert, Aaron Smith, Alex Volfovsky, Nanny Wermuth and Robert Wolpert for helpful comments and encouragement. We also thank the Editor, Associate Editor and the Referees for catching an oversight in our proof in the initial submission, and  NSF and ONR for partial financial support.\bibliographystyle{imsart-nameyear}
\bibliography{cauchy}
\end{document}